\documentclass[12pt,reqno]{amsart}

\newtheorem{lemma}{Lemma}[section]
\newtheorem{definition}{Definition}[section]
\newtheorem{proposition}{Proposition}[section]
\newtheorem{theorem}{Theorem}[section]
\newtheorem{corollary}{Corollary}[section]
\newtheorem{remark}{Remark}[section]

\newtheorem{example}{Example}[section]

\usepackage{amssymb,amsmath,amsthm}
\usepackage{epsfig,graphicx,psfrag}
\usepackage[all]{xy}
\usepackage{enumerate}
\usepackage{animate} 
\usepackage{color}

\begin{document}

\title{A New Characterization of Flat Affine Manifolds and 
	the Associative Envelope of a Left Symmetric Algebra }


\author{Medina A.*,  Saldarriaga, O.**, and Villabon, A.**}

\subjclass[2010]{Primary: 57S20, 54H15; Secondary: 53C07, 17D25  \\ Partially Supported by CODI, Universidad de Antioquia. Project Number 2015-7654.}
\date{\today}

\begin{abstract} This paper deals with flat affine connections on paracompact real manifolds. We give a new characterization of flat affine connections  on  manifolds by means of a certain affine representation  of the Lie group of automorphisms preserving the connection. From the infinitesimal point of view, this representation is determined by the  1-connection form and  the fundamental form of the bundle of linear frames of the manifold. Using the fact that the space of  infinitesimal affine transformations of a flat affine manifold has a natural structure of  finite dimensional associative algebra, we prove the existence of an  associative envelope of the Lie algebra of the Lie group of affine transformations of the manifold. In particular for a Lie group $G$ endowed with a flat affine left invariant connection, we show the existence of a Lie group endowed with a flat affine bi-invariant   connection whose Lie algebra contains the Lie algebra of  $G$. We also exhibit some results about  flat affine manifolds whose group of affine transformations admits a flat affine bi-invariant structure.  The paper is illustrated with several examples.
\end{abstract} 
\maketitle


\vskip5pt
\noindent
 * Universit\'e  Montpellier,  Institute A. Grothendieck, UMR 5149 du CNRS, France and Universidad de Antioquia, Colombia

e-mail: alberto.medina@umontpellier.fr  
\vskip5pt
\noindent
 ** Instituto de Matem\'aticas, Universidad de Antioquia, Medell\'in-Colombia

 e-mails: omar.saldarriaga@udea.edu.co, andresvillabon2000@gmail.com

\section{Introduction}
 
The objects of study of this paper are flat affine paracompact manifolds and their affine  transformations. A well understanding of the category of Lagrangian manifolds assumes a good knowledge of the category of flat affine manifolds (Theorem 7.8 in \cite{W}, see also \cite{F}). Recall that flat affine manifolds with holonomy reduced to $GL_n(\mathbb{Z})$ appear naturally in integrable systems and Mirror Symmetry (see \cite{KS}). 

For simplicity, in what follows $M$ is a connected real n-dimensional manifold, $P=L(M)$ its bundle of linear frames, $\theta$ the fundamental 1-form, $\Gamma$  a linear connection on $P$ of connection form  $\omega$  and  $\nabla$ the covariant derivative on $M$ associated to $\Gamma$. The pair $(M,\nabla)$ is called a flat affine manifold if the curvature and torsion tensors of $\nabla$ are both null. In the case where $M$ is a Lie group and the connection is left invariant, we call it a flat affine Lie group. The correspong Lie algebra Lie$(G)$ will be said a flat affine Lie algebra. A Lie group is flat affine if and only if its Lie algebra is endowed with a left symmetric product compatible with the bracket.

 An affine transformation of $(M,\nabla)$ is a diffeomorphism $f$ of $M$  whose derivative map $f_*:TM\longrightarrow TM$ sends parallel vector fields into parallel vector fields, and therefore geodesics into geodesics (together with its affine parameter). An infinitesimal affine transformation of $(M,\nabla)$ is a smooth vector field $X$ on $M$ whose local 1-parameter groups $\phi_t$ are local affine transformations. We will denote by $\mathfrak{a}(M,\nabla)$ the real vector space of infinitesimal affine transformations of $(M,\nabla)$ and for $\mathfrak{X}(M)$ the Lie algebra of smooth vector fields on $M$. An element $X$ of $\mathfrak{X}(M)$ belongs to $ \mathfrak{a}(M,\nabla)$ if and only if it verifies
\begin{equation}\label{Eq:ecuacionkobayashi} \mathcal{L}_X\circ \nabla_Y -\nabla_Y\circ \mathcal{L}_X=\nabla_{[X,Y]},\quad\text{for all}\quad Y\in \mathfrak{X}(M),\end{equation}
where $\mathcal{L}_X$ denotes the Lie derivative. For a vector field $X\in\mathfrak{X}(M)$, we will denote by $L(X)$ its natural lift on $P$. It is also well known that
\begin{equation} \label{Eq:condition1oninfinitesimalaffinetransformations} X\in\mathfrak{a}(M,\nabla)\text{ if and only if }\mathcal{L}_{L(X)}\omega=0. \end{equation}  This is also equivalent to say that $L(X)$ commutes with every standard horizontal vector field (see Section 2 for the definition of standard horizontal vector field). We will denote by $\mathfrak{a}(P,\omega)$ the set of vector fields $Z$ on $P$ satisfying 
\begin{align} \notag
& Z \text{ is invariant by }R_a\text{ for every }a\in GL_n(\mathbb{R})\\ \label{Eq:infinitesimalaffinetransf}
& \mathcal{L}_Z\theta=0 \\ \notag
&\mathcal{L}_Z\omega=0,
\end{align} 
where $R_a$ denotes the right action of $GL_n(\mathbb{R})$ on $P$. The map $X\mapsto L(X)$ is an isomorphism of Lie algebras from $\mathfrak{a}(M,\nabla)$ to $\mathfrak{a}(P,\omega)$. The vector subspace $\mathsf{aff}(M,\nabla)$ of $\mathfrak{a}(M,\nabla)$ whose elements are complete, with the usual bracket of vector fields,  is the Lie algebra of the group $\mathsf{Aff}(M,\nabla)$ of affine transformations of $(M,\nabla)$ (see \cite{KN}). The image of $\mathsf{aff}(M,\nabla)$ under the isomorphism $X\mapsto L(X)$ will be denoted by $\mathfrak{a}_c(P)$. If $\phi_t$ is the flow of $X$, the flow  of $L(X)$ is the natural lift $L(\phi_t)$ of $\phi_t$ given by $L(\phi_t)(u)=(\phi_t(m),(\phi_t)_{*,m}X_{ij}),$ where $u=(m,X_{ij})$. For the purposes of this work, it is useful to recall that $(P,\omega)$ admits an absolute parallelism, that is, its tangent bundle admits $n^2+n$ sections independent at every point.

\begin{remark} If $X\in\mathfrak{a}(M,\nabla)$ of flow $\phi_t$, the flow  $L(\phi_t)$ of $Z=L(X)$   determines a local isomorphism of $(P,\omega).$ This implies the existence of an open covering $(U_\alpha)_{\alpha\in\mathcal{A}}$ trivializing the bundle $(L(M),M,\pi)$ such that the maps $L(\phi_t):(\pi^{-1}(U_\alpha),\omega)\longrightarrow (\pi^{-1}(\phi_t(U_\alpha)),\omega)$ are isomorphisms. When $Z\in\mathfrak{a}_c(P)$ is complete, the maps $L(\phi_t)$ are global automorphisms of $(P,\omega)$. 
	
Notice that, for every $t$, the map $L(\phi_t)$ preserves the fibration and the horizontal distribution associated to $\omega$.\end{remark}

This paper is organized as follows. In Section \ref{S:newresults} we state a characterization of flat affine manifolds and provide some of its consequences. In the theory of flat affine manifolds appears, in a natural way, a finite dimensional associative  algebra (see Lemma \ref{L:associativity}). This will allow to prove the existence of a natural  associative envelope of the Lie algebra of the Lie group of affine transformations of the manifold. In Section \ref{S:associativeenvelope} we specialize our study to the case of flat affine Lie groups and we exhibit some examples. Section \ref{S:studyofaffinetransformationgroups} deals with the study of the following question. Does the group of affine transformations, $\mathsf{Aff}(M,\nabla)$, of a flat affine manifold $(M,\nabla)$ admit a left invariant flat affine or flat projective structure? This question was answered positively in some particular cases  in the special case of some flat affine Lie groups (see in \cite{MSG}).  

\section{Characterization of flat affine  manifolds} \label{S:newresults}

Before stating the main result of the section, let us denote  by  $A^*$    the fundamental vector field  associated to $A\in\mathfrak{g}=\mathsf{gl}_n(\mathbb{R})$ and by $B(\xi)$ the basic or standard horizontal vector fields ($\theta(B(\xi)):=\xi$  for all $\xi\in\mathbb{R}^n$). Recall that the fundamental form  $\theta$ of $P$ is a tensorial 1-form of type $(GL(\mathbb{R}^n),\mathbb{R}^n)$, that is, it verifies $(R_{a}^*\theta)(Z)=a^{-1}(\theta(Z))$, \ for all \ $a\in\mathsf{GL}_n(\mathbb{R})$ and $Z\in\mathfrak{X}(P)$. 

If  $\Gamma$ is a linear connection over $M$, its connection  form $\omega$ is a  $GL_n(\mathbb{R})$-valued 1-form satisfying  
\begin{align}\label{Eq:omegafixesfundamentals}
&\omega(A^*)=A,\quad\text{for all} \quad A\in\mathsf{gl}_n(\mathbb{R}),\\ \label{Eq:omegaistensorial}
&R_{a}^*\omega=\mathsf{Ad}_{a^{-1}}\omega,\quad\text{for all}\quad a\in GL_n(\mathbb{R}).
\end{align}
It is well known that if $\Omega$ and $\Theta$ are respectively the curvature and torsion forms of the connection $\omega$,  they   verify Cartan's structure equations 
\begin{align*}
d\omega(X,Y)&=-\dfrac{1}{2}\left[\omega(X),\omega(Y)\right]+\Omega(X,Y),\\ 
d\theta(X,Y)&=-\dfrac{1}{2}\left(\omega(X)\cdot\theta(Y)-\omega(Y)\cdot\theta(X)\right)+\Theta(X,Y),
\end{align*}
for all $X,Y\in T_u(P)$, $ u\in P$.  Notice that $\{B(e_1),\dots,B(e_n),E_{11}^*,\dots,E_{nn}^*\}$ is a paralelism of $P$, where $(e_1,\dots,e_n)$ is the natural basis of $\mathbb{R}^n$ and $\{E_{11},\dots,E_{nn}\}$ is the usual basis of $gl_n(\mathbb{R})$.

Considering the simple sheaf $\mathcal{F}$ of base $P$ and fibre the Lie algebra $\mathsf{aff}(\mathbb{R}^n)$ generated by the presheaf $U\rightarrow \mathsf{aff}(\mathbb{R}^n)$, where $U$ is an open set in $P$, we have.

\begin{lemma}  The simple sheaf $\mathcal{F}$ acts on $\mathbb{R}^n$ by 
\begin{equation} \label{EQ:infinitesimalaction} (u,s(u))\cdot w=(u,(v_u,f_u))\cdot w:=v_u+f_u(w) \end{equation}
Moreover, if $\eta:Sec(TP)\longrightarrow Sec(P\times \mathsf{aff}(\mathbb{R}^n))$ is a homomorphism of sheaves of Lie algebras, then $Sect(TP)$ acts  on $\mathbb{R}^n$ via $\eta$.	\end{lemma}
\begin{proof} Recall the construction of the \'etale space corresponding to $\mathcal{F}$. As $\mathcal{F}(u)$ is a direct limit we have $\mathcal{F}(u)=\mathsf{aff}(\mathbb{R}^n)$ for any $u\in P$.

Moreover the sections over $U$ considered as maps from $U$ to $\mathsf{aff}(\mathbb{R}^n)$ are locally constant. Hence the topology of $\mathcal{F}=P\times \mathsf{aff}(\mathbb{R}^n)$ is the product topology of that of $P$ by the discrete topology of $\mathsf{aff}(\mathbb{R}^n)$.

As a consequence, Equation \eqref{EQ:infinitesimalaction} can be written locally as 
\[ s\cdot w=(v,f)\cdot w=v+f(w). \] 
This  gives an infinitesimal action of  $\mathcal{F}$ on $\mathbb{R}^n$.

The last assertion easily follows.
\end{proof}

As a consequence the Lie algebras $Sec(TP),$ $\mathfrak{a}(P)$ and $\mathfrak{a}_c(P)$ also acts infitesimally on $\mathbb{R}^n$.

In these terms we have the following result.

\begin{theorem} \label{T:infinitesimalversion} Let $P=L(M)$ be the principal bundle of linear frames of a manifold $M$, $\theta$ the fundamental form of $P$ and $\Gamma$ a linear connection on $P$ of connection form $\omega$. The following assertions are equivalent
\begin{enumerate}[(i)]
\item The linear connection $\Gamma$ on $P$  is flat affine.
\item The map
$$
\begin{array}{ccc}
\eta:Sec(TP)&\longrightarrow &Sec(P\times\mathsf{aff}(\mathbb{R}^n))\\	Z  &\mapsto         & (\theta(Z),\omega(Z))
\end{array},	$$
is a homomorphisms of sheaves of real Lie algebras.
\end{enumerate} 
\end{theorem}\begin{proof} As $\mathsf{aff}(\mathbb{R}^n)=\mathbb{R}^n\rtimes_{id}\mathsf{gl}(\mathbb{R}^n)$ is the semidirect product of the abelian Lie algebra $\mathbb{R}^n$ and the Lie algebra of commutators of linear endomorphisms of $\mathbb{R}^n$, the map 
$$
\begin{array}{rccl}
\eta:& Sec(TP)&\longrightarrow&Sec(P\times\mathsf{aff}(\mathbb{R}^n))\\
&           Z            &\mapsto    &(\theta(Z),\omega(Z))
\end{array}
$$
is well defined and $\mathbb{R}$-linear.

That $\eta$ is a homomorphism of sheaves of Lie algebras means that
\begin{equation} \label{Eq:homomorphismeta} \eta([Z_1,Z_2])=(\omega(Z_1)\cdot\theta(Z_2)-\omega(Z_2)\cdot\theta(Z_1),[\omega(Z_1),\omega(Z_2)]),\text{ for all }Z_1,Z_2\in \text{Sec}(TP). \end{equation}
In other words, $\omega$ is a linear representation of the Lie algebra $Sec(TP)$ and  $\theta$ is a 1-cocycle relative to this representation.

Suppose that $\Gamma$ is flat and torsion free, then the structure equations for $\omega$ reduce to
\begin{align}\label{Eq:reducedestructureequation1} d\theta(Z_1,Z_2)&=
-\frac{1}{2}(\omega(Z_1)\cdot\theta(Z_2)-\omega(Z_2)\cdot\theta(Z_1))\\ \label{Eq:reducedestructureequation2}
d\omega(Z_1,Z_2)&=
	-\frac{1}{2}[\omega(Z_1),\omega(Z_2)]. \end{align}
It is clear that  \eqref{Eq:homomorphismeta} holds when both vector fields are vertical or if one is vertical and  the other is horizontal.

Now let us suppose that both $Z_1$ and $Z_2$ are horizontal vector fields. As the distribution determined by $\omega$ is integrable, it follows that $[Z_1,Z_2]$ is horizontal and  Equation \eqref{Eq:reducedestructureequation2} implies that $\omega([Z_1,Z_2])=[\omega(Z_1),\omega(Z_2)]$. Also notice that $\omega(Z_1)\cdot\theta(Z_2)-\omega(Z_2)\cdot\theta(Z_2)=0.$ Now, we can suppose that $Z_1$ and $Z_2$ are basic vector fields, that is, $Z_1=B(\xi_1)$ and $Z_2=B(\xi_2)$. Hence  from \eqref{Eq:reducedestructureequation1} we get  that $\theta([Z_1,Z_2])=0$. Therefore assertion (ii) follows from (i).
 

To prove that (ii) implies (i), we will show that  Equation \eqref{Eq:homomorphismeta} implies that the curvature and the torsion vanish, i.e., Equations \eqref{Eq:reducedestructureequation1} and \eqref{Eq:reducedestructureequation2} hold.

As $\eta$ is a homomorphism of sheaves of Lie algebras it follows that for $Z_1, Z_2\in Sec(TP)$ 
	\[ \omega([Z_1,Z_2])=[\omega(Z_1),\omega(Z_2)]\quad\text{ and}\quad \theta([Z_1,Z_2])=\omega(Z_1)\cdot\theta(Z_2)-\omega(Z_2)\cdot\theta(Z_1)\] 
	
The first equality means that the local horizontal distribution on $P$ determined by $\omega$ is  completely integrable, hence  $\Omega$ vanishes.

On the other hand, as the torsion $\Theta$ is a tensorial 2-form on $P$ of type $(GL_n(\mathbb{R}),\mathbb{R}^n)$, we have $\Theta(Z_1,Z_2)=0$ if $Z_1$ or $Z_2$ is vertical. Now,  if  $Z_1=B(\xi_1)$ and $Z_2=B(\xi_2)$  for some $\xi_1,\xi_2\in\mathbb{R}^n$, we obtain
\begin{eqnarray*}
\Theta(Z_1,Z_2) & = & d\theta(Z_1,Z_2)\\ & = & \frac{1}{2}\left(Z_1(\theta(Z_2))-Z_2(\theta(Z_1))-\theta([Z_1,Z_2])\right)\\ & = & \frac{1}{2}\left(Z_1(\xi_2)-Z_2(\xi_1)-\omega(Z_1)\cdot\theta(Z_2)+\omega(Z_2)\cdot\theta(Z_1)\right)\\ & = & 0.
\end{eqnarray*}
Consequently the torsion form $\Theta$ vanishes. 	
\end{proof}

As $P=L(M)$ is parallelizable, the group $K$ of transformations of this parallelism is a Lie group of dimension at most  $\dim P=n^2+n$. More precisely, given $\sigma\in K$ and for any $u\in P$, the map $\sigma\mapsto \sigma(u)$ is injective and its image $\{\sigma(u)\mid \sigma\in K\}$ is a closed submanifold of $P$. The submanifold structure of $\{\sigma(u)\mid \sigma\in K\}$ turns $K$  into a Lie transformation group of $P$ (see \cite{SK}).  The group $Aut(P,\omega)$, of diffeomorphisms of $P$ preserving $\omega$ (therefore $\theta$), is a  closed Lie subgroup of $K$, so it is a Lie group of transformations of $P$.


\begin{remark} \label{R:affineatlas}
The group $Aut(L(\mathbb{R}^n),\omega^0)$  of diffeomorphisms of $L(\mathbb{R}^n)$ preserving the usual connection $\omega^0$ on $\mathbb{R}^n$ is isomorphic to the  group $\mathsf{Aff}(\mathbb{R}^n,\nabla^0)$, called the classic affine group. Also recall that  having a flat affine structure $\nabla$ on $M$ is equivalent to have a smooth atlas whose change of coordinates are elements of $Aut(L(\mathbb{R}^n),\omega^0)$.	
\end{remark}

The following is a consequence of Theorem \ref{T:infinitesimalversion}. 
\begin{corollary} Let $\Gamma$ be a linear connection of connection form $\omega$   \begin{enumerate}[a.]
\item If $\Gamma$ is a metric connection, i.e., $\nabla g=0$ with $g$ a pseudo-Riemannian metric, then $\Gamma$  is flat affine if and only if the homomorphism $\eta$ of the Theorem \ref{T:infinitesimalversion} takes values in the Lie algebra $e_{(p,q)}(n)=\mathbb{R}^n\rtimes o_{(p,q)}$, of the pseudo-Euclidian group $E_{(p,q)}(n)=\mathbb{R}^n\rtimes O_{(p,q)}$, where $(p,q)$ denotes the signature of the pseudo-Riemannian metric.
\item If the connection $\Gamma$ preserves a volume form, then $\Gamma$ is flat affine if and only if  the homomorphism $\eta$ of the Theorem \ref{T:infinitesimalversion} takes values in the Lie algebra $\mathbb{R}^n\rtimes sl_n(\mathbb{R})$ of the group $\mathbb{R}^n\rtimes SL_n(\mathbb{R}).$ 
		
\noindent 
In particular, if $\Gamma$ preserves a symplectic form  $\sigma$, i.e. $\nabla\sigma=0$, then $\Gamma$ is flat affine if and only if  the homomorphism $\eta$ of Theorem \ref{T:infinitesimalversion} takes values in the Lie algebra $\mathbb{R}^{n}\rtimes sp_{n}(\mathbb{R})$  of $\mathbb{R}^{n}\rtimes Sp_{n}(\mathbb{R}),$ where $Sp_{n}(\mathbb{R})$ is the group of linear simplectomorphisms.\end{enumerate}
\end{corollary}
\begin{proof} The result follows by reduction of the principal bundle $L(M)$ to the respective subbundle with structural group $O_{(p,q)}$, $SL_n(\mathbb{R})$ and $Sp_{n}(\mathbb{R})$. In every case, the homomorphism $\eta$ takes values in the sheaf of Lie algebras $Sec(P\times (\mathbb{R}^n\rtimes o_{(p,q)}))$, $Sec(P\times (\mathbb{R}^n\rtimes sl_n(\mathbb{R}))$ and  $Sec(P\times (\mathbb{R}^n\rtimes sp_{n}(\mathbb{R}))$, respectively. 
\end{proof}

\begin{remark} 
If $\Gamma$ is flat affine and $Z=L(X)\in\mathfrak{a}(P,\omega)$, the local isomorphism $L(\phi_t)$ preserves the foliation defined by $\omega$, where $\phi_t$ is the flow of $X$. That is, the natural lift $L(\phi_t)$ of $\phi_t$ is a local isomorphism of $(P,\omega)$. 
\end{remark}

The groups $Aut(P,\omega)$ and  $\mathsf{Aff}(M,\nabla)$ are  isomorphic. We denote by  $Aut(P,\omega)_0$ the unit component of $Aut(P,\omega)$ and by $\widehat{Aut}(P,\omega)_0$ its universal covering Lie group.

\begin{example} \label{Ex:calculationofaffinegroups} Let  $M_1$, $M_2$ and $M_3$  be the plane without one, two and three non colineal points respectively endowed with $\nabla_i$ the connection $\nabla^0$ restricted to $M_i$, $i=1,2,3$.  Let us suppose that the  points removed are $p_1=(0,0)$, $p_2=(0,1)$ and $p_3=(1,0)$. An easy calculation shows that the corresponging groups of affine transformations of $M_i$ relative to $\nabla_i$ are given by
	\[  \mathsf{Aff}(M_1,\nabla_1)=\{  F:M_1\longrightarrow M_1\mid F(x,y)=(ax+by,cx+dy)\quad\text{suh that}\quad ad-bc\ne0\}.     \]
	\[  \mathsf{Aff}(M_2,\nabla_2)=\{  F\in Diff(M_2)\mid F(x,y)=(ax,bx+y)\text{ or }  F(x,y)=(ax,bx-y+1), \ a\ne0\}.     \]
The connected component of the unit of the group $ Aut(L(M_2),\omega_2)$, of automorphisms of $P=L(M_2)$ preserving the connection form $\omega_2$ associated to $\nabla_2$ is given by
\[\left\{F_{a,b}\in Diff(P)\ \bigg|\  F_{a,b}\left(x,y,\left[\begin{matrix}X_{11}&X_{12}\\X_{21}&X_{22}\end{matrix}\right]\right)=\left(ax,bx+y,\left[\begin{matrix}aX_{11}&aX_{12}\\bX_{11}+X_{12}&bX_{12}+X_{22}\end{matrix}\right]\right)\right\}.  \]
The group $\mathsf{Aff}(M_3,\nabla_3)$ is discrete, its elements are the set of affine transformations of $(\mathbb{R}^2,\nabla^0)$ permuting the points  $p_1,\ p_2$ and $p_3$. 
	
	Moreover, consider $\sigma_i$ the restriction of the usual symplectic form $\sigma^0=dx\wedge dy$ on $M_i$, $i=1,2,3$. It is clear  that the connection $\nabla^0$  is symplectic with respect to $\sigma^0$. Consequently the manifolds $(M_i,\nabla_i,\sigma_i)$, $i=1,2,3$, are flat affine symplectic manifolds. Then the corresponding groups of affine symplectomorphisms are given by 
	\[ Sp\mathsf{Aff}(M_1,\sigma_1,\nabla_1)=\{  F:M_1\longrightarrow M_1\mid F(x,y)=(ax+by,cx+dy)\quad\text{with}\quad ad-bc=1\}.     \]
	\[  Sp\mathsf{Aff}(M_2,\sigma_2,\nabla_2)=\{  F\in Diff(M_2)\mid F(x,y)=(x,bx+y)\text{ or }  F(x,y)=(-x,bx-y+1)\}.     \]
	Finally, the elements of the  group $Sp\mathsf{Aff}(M_3,\nabla_3)$ are affine transformations of $(\mathbb{R}^2,\nabla^0)$ permuting the points  $p_1,\ p_2$ and $p_3$ and preserving the orientation determined by $\sigma^0$. 
\end{example}

The following consequence of Theorem \ref{T:infinitesimalversion}  is obtained by applying Lie's third theorem and Cartan's Theorem  (see \cite{Che} for details on the proof of  Lie's third theorem).

\begin{corollary} \label{C:integrationofeta} Let $(M,\nabla)$ be a flat affine connected manifold    of connection form $\omega$. Then there exists a Lie group homomorphism $\rho:\widehat{Aut}(P,\omega)_0\longrightarrow Aut(L(\mathbb{R}^n),\omega^0)$ determined by $(\theta,\omega)$.
\end{corollary}

Using techniques and results of  Douady A. and Lazard M. \cite{DL}, one can  prove the following.  

\begin{proposition} The fibre bundle in Lie algebras $P\times \mathfrak{a}_c(P)$ (respectively $P\times \mathsf{aff}(\mathbb{R}^n)$) is isomorphic to the Lie algebra of a Hausdorff Lie group $G$ over $P$ such that for any $u\in P$, the fiber over $u$ is $\widehat{Aut}(P,\omega^+)_0$ (respectively $\widehat{\mathsf{Aff}}(\mathbb{R}^n)_0$). If we want that the fiber over $u$ be $Aut(P,\omega)_0$ (respectively $\mathsf{Aff}(\mathbb{R}^n)_0$), the topology of $G$ is not necessarily  Hausdorff.
\end{proposition}

\begin{remark} Let $\Gamma$ be a flat affine connection on $P=L(M)$, then
\begin{enumerate}[(a)]
\item The set of vectors fields $\{B(e_1),\dots,B(e_n)\}$ determine a  horizontal submanifold of $P$ (relative to $\omega$). 
\item The connection $\Gamma$ is flat affine if and only if for every $x\in M$ there exists a coordinate system $(U;x_1,\dots,x_n)$ so that $\nabla_{\partial_i}\partial_j=0$.
\end{enumerate}
\end{remark}




The Lie algebra homomorphism $\eta$ of Theorem \ref{T:infinitesimalversion} could be integrated, in some particular cases, into a Lie group homomorphism. For instance using the affine development relative to $\nabla$ (\cite{E}, \cite{JLK} and for more details see \cite{Sh}). We have.

\begin{proposition} If $M$ is simply connected and $\Gamma$ is flat affine, there exists  a Lie group homomorphism $\rho:Aut(P,\omega)\longrightarrow Aut(L(\mathbb{R}^n),\omega^0)$  with derivative $\rho_{*,Id}=(\omega,\theta)$. 
\end{proposition}

More particularly we have

\begin{proposition} If $M$ is connected and $\Gamma$ is geodesically complete, there is an isomorphism of Lie groups $\rho:Aut(P,\omega)\longrightarrow Aut(L(\mathbb{R}^n),\omega^0)$  with derivative $\rho_{*,Id}=(\omega,\theta)$. 
\end{proposition}

\begin{remark} Notice that a local expression of the homomorphism $\rho$ of the previous proposition can be obtained by integration of the homomorphism of Lie algebras $\eta:\mathfrak{a}_c(P)\longrightarrow \mathsf{aff}(\mathbb{R}^n)$.
\end{remark}

\begin{example} Consider the manifold $M=\mathbb{R}^2\setminus\{(0,0),(0,1)\}$ endowed with the (no complete) connection $\nabla$  restriction of the usual connection of $\mathbb{R}^2$ to $M$ and let $G(\mathfrak{a}_c(P))$ be the connected and simply connected real Lie group of Lie algebra $\mathfrak{a}_c(P)$. A direct calculation shows the homomorphism $\eta$ of the previous remark integrates to the Lie group  homomorphism $$\begin{matrix}\rho:&G(\mathfrak{a}_c(P))&\longrightarrow& Aff(\mathbb{R}^2)\\ &F_{a,b}&\mapsto&\rho(F_{a,b})\end{matrix}$$  where $\rho(F_{a,b})$ is given by
\[  \frac{1}{D}\left[\begin{matrix} aX_{11}X_{22}-bX_{11}X_{12}-X_{12}X_{21}&(a-1)X_{12}X_{22}-bX_{12}^2&(a-1)xX_{22}-bxX_{12}\\(1-a)X_{11}X_{21}+bX_{11}^2& -aX_{12}X_{21}+bX_{11}X_{12}+X_{11}X_{22}&(1-a)xX_{21}+bxX_{11}\\0&0&D         \end{matrix}\right]  \] 
with $F_{a,b}$ as in Example \ref{Ex:calculationofaffinegroups}, $\left[\begin{matrix} X_{11} & X_{12}\\ X_{21}&X_{22}\end{matrix} \right]$ a system of local coordinates of $GL_2(\mathbb{R})$,   $D=X_{11}X_{22}-X_{12}X_{21}$ and $(x,y)$ local coordinates of $M$. 

\end{example}

%



Recall that $\omega$ is flat affine if and only if its affine holonomy group is discrete (see \cite{KN}). Nevertheless, there are flat affine manifolds whose linear holonomy group is not discrete. If $\omega$ is a complete flat affine  connection, its linear holonomy group is  discrete. 

There exists topological obstructions to the existence of a  flat affine connection (see  \cite{S}). In particular, a closed manifold with finite fundamental group does not admit flat affine structures (see \cite{Aus}).





If $M=G$ is a Lie group, we will identify its Lie algebra  Lie$(G)=\mathfrak{g}$ with the real vector space of left invariant vector fields on $G$ and also with the tangent space at the unit $\varepsilon\in G$. For any $x\in T_\varepsilon G$, we denote by $x^+$ (respectively by $x^-$) the left (right) invariant vector field determined by $x$.
The group  $G$ acts on itself on the left (respectively right). The left (respectively right) action of  $\sigma\in G$  will be denoted by $L_\sigma$ (respectively $R_\sigma$). These actions naturally lift to actions of $G$ on $P=L(G)$ given by $\psi_{1,\sigma}(u)=\sigma\cdot u= \left(\sigma\tau,(L_\sigma)_{*,\tau}(X_{ij})\right)$ (respectively   $\psi_{2,\sigma}(u)=u\cdot \sigma=(\tau\sigma,(R_\sigma)_{*,\tau}(X_{ij}))$
where $u=(\tau,(X_{ij}))$. A linear connection $\Gamma$ on $G$ is said left invariant  (respectively right invariant) if the action $\psi_1$ (respectively  $\psi_2$) preserves the  horizontal distribution. The connection is bi-invariant if the  horizontal distribution is preserved by  $\psi_1$ and $\psi_2$. 

If $\nabla=\nabla^+$ is left invariant, it is well known that $(G,\nabla^+)$ is flat affine if and only if the product $x\cdot y=\nabla^+_{x^+}y^+$ turns $\mathfrak{g}$ into a left symmetric algebra (see for instance \cite{M}).

The existence of a left invariant flat affine connection on a Lie group $G$, is equivalent to the existence of an \'etale affine representation of $G$. This means, that there exists a homomorphism of Lie groups $\tau:G\longrightarrow \mathsf{Aff}(\mathbb{R}^n)$ with an open orbit and discrete isotropy. 


In these terms we get the following consequence of Theorem \ref{T:infinitesimalversion}  and Corollary \ref{C:integrationofeta}. 
 
\begin{remark} It is clear that having a left invariant connection on a Lie group $G$ is equivalent to have a bilinear product on  $Lie(G)$. In what follows we can suppose, if necessary, the Lie group $G$  simply connected.
\end{remark}
 
\begin{theorem} \label{T:rhointheLiegroupscase} Let $G$ be a connected Lie group endowed with a left invariant connection $\nabla^+$ of connection form $\omega^+$. The following assertions are equivalent
\begin{enumerate}	\item[1] The connection $\omega^+$ is flat affine.
\item[2] There exists a unique Lie group homomorphism  $\rho:\widehat{Aut}(L(G),\omega^+)_0\longrightarrow Aut(L(\mathbb{R}^n),\omega^0)$ with derivative  $\rho_{*,Id}=(\theta,\omega)$. 
\item[3] There exists an affine \'etale representation $\rho':G^{op}\longrightarrow\mathsf{Aff}(\mathbb{R}^n,\nabla^0)$. 
\item[4] The Lie group $G^{opp}$, opposite of $G$, acts on $(G,\nabla^+)$ by affine transformations  via the map $\sigma\mapsto\sigma^{-1}.$ 
\end{enumerate}  \end{theorem}\begin{proof}
Suppose that $\Gamma$ is flat affine.  From Theorem \ref{T:infinitesimalversion}, for $u\in L(G)$ fixed, there exists a hommomorphism of Lie algebras $\eta':\mathfrak{X}(P)\longrightarrow aff(\mathbb{R}^n)$ definde by $\eta'(Z)=(\theta_u(Z),\omega_u(Z_u))$. Hence, the restriction $\eta'':=\eta'/_{\mathfrak{a}_c(P)}:\mathfrak{a}_c(P)\longrightarrow aff(\mathbb{R}^n)$ of $\eta'$ to $\mathfrak{a}_c(P)$,  is a Lie algebra homomorphism. Using Cartan's Theorem and Lie's third theorem, we get a  Lie group homomorphism  $\rho:\widehat{Aut}(L(G),\omega^+)_0\longrightarrow Aff(\mathbb{R}^n)$. Also notice that $\{\eta''(Z)(0)\mid Z\in\mathfrak{a}_c(P)\}=\mathbb{R}^n$, consequently $\Lambda$ admits an open orbit.

Now suppose that there is a Lie group homomorphism $\rho:\widehat{Aut}(L(G),\omega^+)_0\longrightarrow Aut(L(\mathbb{R}^n),\omega^0)$ with a point of open orbit so that $\rho_{*,Id}=(\theta,\omega)$. That is $\rho_{*,Id}=\eta''$ as above. As right invariant vector fields on $G$ are infinitesimal affine transformations, we have that $\mathfrak{g}^{op}$ can be considered as a Lie subalgebra of $\mathfrak{a}_c(L(G))$, where $\mathfrak{g}=$Lie$(G)$. Therefore the restriction $\eta'''=\eta''/_{\mathfrak{g}^{op}}:\mathfrak{g}^{op}\longrightarrow aff(\mathbb{R}^n)$ is also a Lie algebras homomorphism. Using Cartan's Theorem and Lie's third theorem we get a Lie group homomorphism $\rho':G^{op}\longrightarrow  Aff(\mathbb{R}^n)$. 

For $x\in T_\epsilon G$, denote $x^-$ the right invariant vector field determined by $x$ and by $L(x^-)$ its natural lift to $L(G)$. We conclude that $\rho'$ defines an affine \'etale action of $G^{op}$ on $\mathbb{R}^n$, by noticing that  $\{\eta'''(L(x^-))(0)\mid x\in T_\epsilon G\}=\mathbb{R}^n$.

The statement 3 implies 1 is proved in \cite{M} (see also \cite{JLK}). 
\end{proof}


\section{The associative envelope of a finite dimensional left symmetric algebra} \label{S:associativeenvelope}

Given a flat affine manifold $(M,\nabla)$, we show the existence of a simply connected  Lie group endowed with a flat affine bi-invariant connection whose Lie algebra contains $\mathfrak{a}_c(M,\nabla)$ as a Lie subalgebra (see Theorem \ref{T:envelopeofaflataffinemanifold}). We specialize this result to the case where $M=G$ is a flat affine Lie group. In this special case we show that the Lie algebra of the simply connected  Lie group $Env(G,\nabla^+)$   contains the Lie algebra of $G$ (see Theorem \ref{T:existenceofenvelope}).

For this purpose we recall some known facts and introduce some notation. Let $M$ be an $n$ dimensional manifold with a linear connection $\Gamma$ and corresponding covariant derivative   $\nabla$, consider the product \begin{equation}\label{Eq:productinducedbyaconnection}XY:=\nabla_XY,\quad\text{for}\quad X,Y\in \mathfrak{X}(M).\end{equation} 
If the curvature  tensor relative to $\nabla$ is identically zero, this product verifies the condition
\[ [X,Y]Z=X(YZ)-Y(XZ), \quad\text{for all}\quad X,Y,Z\in \mathfrak{X}(M).\]
If both torsion and curvature vanish identically we have
\begin{equation} \label{Eq:leftsymmetricproduct} (XY)Z-X(YZ)=(YX)Z-Y(XZ), \quad\text{for all}\quad X,Y,Z\in \mathfrak{X}(M).\end{equation} 
A vector space endowed with a bilinear product satisfying Equation \eqref{Eq:leftsymmetricproduct} is called a left symmetric algebra. In this case the product given by $[X,Y]_1:=XY-YX$ defines a Lie bracket on the space. If $\nabla$ is torsion free, this Lie bracket agrees with the usual Lie bracket of $\mathfrak{X}(M)$, that is \begin{equation}\label{equliebrakettorsionfree}
	[X,Y]=\nabla_XY-\nabla_YX
\end{equation}
Recall that product \eqref{Eq:productinducedbyaconnection} satisfies $(fX)Y=f(XY)$ and $X(gY)=X(g)Y+g(XY)$, for all $f,g\in C^\infty(M,\mathbb{R})$, i.e., the product defined above is $\mathbb{R}$-bilinear and $C^\infty(M,\mathbb{R})$-linear in the first component. 

 In what follows, given an associative or a left symmetric algebra $(\mathcal{A},\cdot)$, we will denote by $\mathcal{A}_-$ the Lie algebra of commutators of $\mathcal{A}$, i.e., the Lie algebra with bracket given by \[ [a,b]=a\cdot b-b\cdot a.\]
The next result may be already known by some experts on the area, see \cite{T}.

\begin{lemma}\label{L:associativity} Let $(M,\nabla)$ be an $n$-dimensional flat affine manifold. Then Product \eqref{Eq:productinducedbyaconnection} induces a natural structure of associative algebra, of dimension less than or equal to $n^2+n$, on the vector space $\mathfrak{a}(M,\nabla)$  whose commutator is the Lie bracket of vector fields on $M$.  In particular if $M=G$ is a Lie group and $\nabla=\nabla^+$ is left invariant, $\mathfrak{a}(G,\nabla^+)_-$   contains $\mathfrak{g}^{opp}$  as a  subalgebra, where $\mathfrak{g}^{opp}$ is the opposite Lie algebra of $\mathfrak{g}=Lie(G)$.
\end{lemma}
\begin{proof}
Since $\nabla$ is flat affine, it follows from \eqref{Eq:ecuacionkobayashi} that a smooth vector field $X$ is an infinitesimal affine transformation if and only if 
\begin{equation} \label{Eq:infinitesimalaffinetransformationsonflataffineconnections} \nabla_{\nabla_YZ}X=\nabla_Y\nabla_ZX,\end{equation}
for all $Y,Z\in\mathfrak{X}(M)$. This equality  implies that $\nabla_XY\in\mathfrak{a}(M,\nabla),$ whenever $X,Y\in \mathfrak{a}(M,\nabla)$. It follows from \eqref{Eq:infinitesimalaffinetransformationsonflataffineconnections} that the product $X Y=\nabla_XY$ is an associative product on $\mathfrak{a}(M,\nabla).$ Moreover, from \eqref{equliebrakettorsionfree}  the commutator of $\mathfrak{a}(M,\nabla)$ is the Lie bracket of vector fields on $M$. 
	
In particular if $M=G$ is a Lie group of Lie algebra $\mathfrak{g}$ and $\nabla^+$ is left invariant and flat affine,  the real vector space $\mathfrak{g}$ of right invariant vector fields on $G$ is a subspace of $\mathsf{aff}(G,\nabla^+)$. Hence from \eqref{equliebrakettorsionfree} we get  that $\mathfrak{g}^{op}$ is a Lie subalgebra of  $\mathfrak{a}(G,\nabla^+)_-$. 
\end{proof}


The following result is useful in what follows (see for instance \cite{M} or  \cite{BM}). 

\begin{theorem}\label{T:associativeimplybiinvariantstructures}  Let $G$ be a Lie group and $\mathfrak{g}$ its Lie algebra. The group $G$ admits a flat affine bi-invariant structure if and only if  $\mathfrak{g}$ is the underlying Lie algebra of an  associative algebra so that $[a,b]=ab-ba$, for all $a,b\in\mathfrak{g}$. 
\end{theorem}

We have the following consequences of Lemma \ref{L:associativity}.

\begin{theorem}\label{T:envelopeofaflataffinemanifold} Given a flat affine manifold $(M,\nabla)$ there exists a  connected   Lie group $G$ endowed with a flat affine  bi-invariant structure containing a connected Lie subgroup $H$ locally isomorphic to $\mathsf{Aff}(M,\nabla)$.
\end{theorem}
\begin{proof} From Lemma \ref{L:associativity}, the real vector space $\mathfrak{a}(M,\nabla)$ is an associative algebra under the product determined by $\nabla$. Let $E$ be the smallest associative subalgebra of $\mathfrak{a}(M,\nabla)$ containing $\mathsf{aff}(M,\nabla)$ and $\widehat{\mathcal{A}}=E\oplus\mathbb{R}1$ the associative algebra obtained from $E$ by adjoining a unit element $1$. Consider $U(\widehat{\mathcal{A}})$ the group of units of $\widehat{\mathcal{A}}$, this is open and dense in $\widehat{\mathcal{A}}$. Let 
$$\overline{G}:=\left\{u\in U(\widehat{\mathcal{A}})\mid u=1+a,\ \text{with}\ a\in E  \right\}$$
and $G=\overline{G}_0$ the connected component of the unit of $G$. Then the Lie group $G$ verifies the conditions of the theorem (for more details see \cite{BM}).

As $\nabla$ is flat affine, the space $\mathfrak{a}_c(M,\nabla)$ is a Lie subalgebra of $E_-$. Consequently there exists a connected Lie subgroup $H$ of $G$ of Lie algebra $\mathfrak{a}_c(M,\nabla)$. On the other hand, $\mathfrak{a}_c(M,\nabla)$ is also the Lie algebra of $\mathsf{Aff}(M,\nabla)$, hence the groups $H$ and $\mathsf{Aff}(M,\nabla)$ are locally isomorphic. 
\end{proof}
Let $\mathcal{A}^{op}$ be the opposite algebra of a finite dimensional associative algebra $\mathcal{A}$. We have

\begin{theorem} \label{T:existenceofenvelope}
If $(\mathfrak{g},\cdot)$ is an $n$-dimensional left symmetric algebra over $\mathbb{K}=\mathbb{R}$ or $\mathbb{C}$, then there exists:
\begin{enumerate}
\item[1.] A unique connected and  simply connected flat affine Lie group $(G,\nabla^+)$ whose Lie algebra is $Lie(G)=\mathfrak{g}_-$, the Lie algebra  of commutators of $\mathfrak{g}$.
\item[2.] An associative algebra $\mathcal{A}$ so that $\mathfrak{g}_-$ is a Lie subalgebra  $\mathcal{A}_- $, the Lie algebra of commutators of $\mathcal{A}$. 	Moreover $\mathcal{A}^{op}$ is the subalgebra  of the associative algebra $\mathfrak{a}(G,\nabla^+)$ generated by $\mathsf{aff}(G,\nabla^+)$.
	\end{enumerate}
\end{theorem} 
\begin{proof} Let $G$ be the connected and simply connected Lie group of Lie algebra $\mathfrak{g}_-$. The product on the left symmetric algebra $\mathfrak{g}$ determines a flat affine left invariant connection $\nabla^+$ on $G$. From Lemma \ref{L:associativity}, the connection $\nabla^+$ induces an associative product  over $\mathfrak{a}(G,\nabla^+)$ compatible with the Lie bracket of vector fields on $G$, that is
	\[ [X,Y]=XY-YX.\] 
Let $\mathcal{B}$ be the  subalgebra of the associative algebra $\mathfrak{a}(G,\nabla^+)$ generated by $\mathsf{aff}(G,\nabla^+)$. Then 
$\mathcal{A}$ is precisely the opposite associative algebra of $\mathcal{B}$. As the right invariant vector fields on $G$ are complete infinitesimal affine transformations, it follows that $\mathfrak{g}_-$ is a Lie subalgebra of $\mathcal{A}_-$.
\end{proof}

This theorem  motivates the following 

\begin{definition}\label{D:associativeenvelope} Under the terms of the previous theorem we define the following.
\begin{enumerate} 
\item[(a)] The associative envelope $Env(\mathfrak{g})$ of a left symmetric algebra $(\mathfrak{g},\cdot)$ is the opposite of the associative subalgebra of $\mathfrak{a}(G,\nabla^+)$ generated by $\mathsf{aff}(G,\nabla^+)$. That is, $Env(\mathfrak{g})$ is the opposite of the smallest associative subalgebra of  $\mathfrak{a}(G,\nabla^+)$ containing  $\mathsf{aff}(G,\nabla^+)$.
\item[(b)] Any Lie group of Lie algebra $Env(\mathfrak{g})_-$ will be called an enveloping Lie group of $(G,\nabla^+)$.
\item[(c)] The associative subalgebra algebra $Env(M,\nabla)$ of $\mathfrak{a}(M,\nabla)$ spanned by $\mathfrak{a}_c(M,\nabla)$, with $(M,\nabla)$ a flat affine manifold, is called the associative envelope of $\mathfrak{a}_c(M,\nabla)$.\\ 
Any Lie group of Lie algebra $Env(M,\nabla)_-$  will be called an enveloping Lie group of $\mathsf{Aff}(M,\nabla)$.
	\end{enumerate}	
\end{definition}

\begin{remark} Let $G$ be a flat affine Lie group.
\begin{enumerate}[i.]
\item Any enveloping Lie group of  $G$ is endowed with a flat affine bi-invariant connection.
\item Although the elements of the associative envelope $Env(\mathfrak{g})$ of the left symmetric algebra  $\mathfrak{g}=$Lie$(G)$ are differential operators of order less than or equal to 1 on $G$, the associative envelope is not a subalgebra of the universal enveloping algebra of $\mathfrak{g}$. 
\end{enumerate} \end{remark}

\begin{example} Let $M=\mathbb{R}^2\setminus\{(0,0) \}$ and $\nabla$ the restriction to $M$ of the usual connection $\nabla^0$. We are going to show that $G=\mathsf{Aff}(M,\nabla)$ is an enveloping Lie group of itself. First notice that the set $\left\{x\dfrac{\partial}{\partial x},y\dfrac{\partial}{\partial x},x\dfrac{\partial}{\partial y},y\dfrac{\partial}{\partial y},\dfrac{\partial}{\partial x},\dfrac{\partial}{\partial y} \right\}$ is a linear basis of the real vector space $\mathfrak{a}(M,\nabla)$. It is also easy to verify that $\left\{x\dfrac{\partial}{\partial x},y\dfrac{\partial}{\partial x},x\dfrac{\partial}{\partial y},y\dfrac{\partial}{\partial y} \right\}$ is  a basis of the real vector space $\mathfrak{a}_c(M,\nabla)$. On the other hand, it is verified that $\mathfrak{a}_c(M,\nabla)$ is a subalgebra  of the associative algebra $\mathfrak{a}(M,\nabla)$. Consequently $G=\mathsf{Aff}(M,\nabla)$ is endowed of a flat affine bi-invariant connection determined by $\nabla$. 
\end{example}

\begin{example} \label{Ex:productoR1}
	Let $\nabla^+$ be the flat affine left invariant connection  on $G=\mathsf{Aff}(\mathbb{R})$ defined by
	\begin{equation} \label{Eq:productodeterminadopornablacasoR1} \nabla^+_{e_1^+}e_1^+=2e_1^+,\qquad\nabla^+_{e_1^+}e_2^+=e_2^+,\qquad\nabla^+_{e_2^+}e_1^+=0,\quad \text{and}\quad\nabla^+_{e_2^+}e_2^+=e_1^+.\end{equation}
	A direct calculation shows that a linear basis of $\mathfrak{a}(G,\nabla^+)$ is given by the following vector fields
	\[ e_1^-=x\frac{\partial}{\partial x}+y\frac{\partial}{\partial y},\quad e_2^-=\frac{\partial}{\partial y},\quad C_3=\frac{1}{x}\frac{\partial}{\partial x},\quad C_4=\frac{y}{x}\frac{\partial}{\partial x},\quad  C_5=\left(x+\frac{y^2}{x}\right)\frac{\partial}{\partial x}\]\[\text{and}\qquad
	C_6=\left(-xy-\frac{y^3}{x}\right)\frac{\partial}{\partial x}+(x^2+y^2)\frac{\partial}{\partial y},\]
	where $e_1^-$ and $e_2^-$ denote the right invariant vector fields. As the connection is left invariant, the vector fields  $e_1^-$ and $e_2^-$ are complete. Moreover, it can be checked that no real linear combination of the fields $C_3,$ $C_4$, $C_5$ and $C_6$ is complete.  Consequently $(e_1^-,e_2^-)$ is a linear basis of $\mathsf{aff}(G,\nabla^+)$. 
	
	On the other hand, the multiplication table of  the product defined by $\nabla^+$, i.e., the product $XY=\nabla_X^+Y$, on the basis of $\mathfrak{a}(G,\nabla^+)$ displayed above  is given by

	\begin{equation} \label{Tab:tabla1}
		\begin{array}{c|c|c|c|c|c|c}  & e_1^- &e_2^-&C_3&C_4&C_5&C_6\\\hline e_1^-&e_1^-+C_5& C_4&0&C_4&2C_5&2C_6\\\hline e_2^-& e_2^-+C_4&C_3&0&C_3&2C_4&2e_1^- -2C_5\\\hline C_3& 2C_3 & 0 &0&0&2C_3&2e_2^--2C_4\\\hline C_4&2C_4 &0 &0 &0&2C_4&2e_1^--2C_5\\\hline C_5&2C_5 &0 &0 &0 &2C_5 &2C_6\\\hline C_6& C_6&C_5 &0 & C_5&0&0 
		\end{array} 
	\end{equation}
	
	It follows from Table \ref{Tab:tabla1} that the real associative subalgebra  of $\mathfrak{a}(G,\nabla^+)$ generated by $\{e_1^-,e_2^-\}$ has linear basis $(e_1^-,e_2^-,C_3,C_4,C_5)$.That is, $Env(\mathfrak{g})$ is the real 5-dimensional associative algebra with linear basis $(e_1^-,e_2^-,C_3,C_4,C_5)$ and the opposite product of Table \eqref{Tab:tabla1}. 
	
	The  animation below shows the lines of  flow of each of the vector fields $e_1^-,e_2^-,C_3,C_4,C_5$ and $C_6$ with the initial condition $(1.5,-1)$. The fact that the flows of the vector fields $C_3,C_4,C_5$ and $C_6$ cross the boundary to the outside of the orbit of $(0,0)$ determined by the affine \'etale representation relative to $\nabla$, corresponds to the fact that the vector fields are not complete in $G_0$. (To play the animation click on the image).
	\begin{center}
		\animategraphics[controls,loop,height=7cm]{4}{imagen}{1}{46}
	\end{center}
\end{example}

\begin{remark} The reader can verify that the simply connected Lie group of Lie algebra  $Env(\mathfrak{g})_-$ of the previous example is isomorphic to the group   $\mathbb{E}:=(\mathbb{R}^2\rtimes_{\theta_1} \mathbb{R})\rtimes_{\theta_2} G$, where $G=\mathsf{Aff}(\mathbb{R})_0$ is the connected component  of the unit of $\mathsf{Aff}(\mathbb{R})$ and the actions $\theta_1$ and $\theta_2$ are respectively given by
\[   \theta_1(t)(x,y)=(e^{2t}x,e^{2t}y)\quad\text{ and }\quad  \theta_2(x,y)(z_1,z_2,z_3)=\left(x^2z_1-xyz_2+y^2z_3,xz_2-2yz_3,z_3\right) \] 
The Lie group $\mathbb{E}$ is also isomorphic to the semidirect product $\mathcal{H}_3\rtimes_\rho \mathbb{R}^2$ of the additive group $\mathbb{R}^2$ acting on the 3-dimensional Heisenberg group where $\rho$ is given by 
\begin{equation} \label{Eq:actiononheisenberg} \rho(a,b)(x,y,z)=\left(e^{a}x,e^{2a+2b}y,e^a(e^{2b}-1)x+e^{a+2b}z\right).\end{equation}
	
\end{remark}

At this point we have some questions for which we do not have an answer.

\noindent \textbf{Questions.}\begin{enumerate}
\item[1.]  To determine a transformation group $E$ of diffeomorphisms of $P=L(M)$ so that $\mathsf{Lie}(E)=Env(\mathfrak{a}_c(P))$ in the case  when $\dim(Env(\mathfrak{a}_c(P)))<n^2+n$. \vskip3pt
\item[2.] Is it possible to realize an enveloping Lie group  of a flat affine Lie group $(G,\nabla^+)$ as a group of transformations of $L(G)$ (eventually as a subgroup of the Lie group $K$)? 
\end{enumerate}
\noindent
If $\mathfrak{a}_c(P)=\mathfrak{a}(P,\omega)$ the answer to the first question is positive. When $\dim(Env(\mathfrak{a}_c(P)))<n^2+n$, the method described on Theorem \ref{T:envelopeofaflataffinemanifold} could give an answer to these questions. \vskip5pt

%

The following example describes the  generic case of flat affine connections on $\mathsf{Aff}(\mathbb{R})$.

\begin{example} \label{Ex:genericcase}
	Let us consider the family of left symmetric products on $\mathfrak{g}=\mathsf{aff}(\mathbb{R})$  given by
	$$\begin{array}{c|c|c}
	\cdot&e_1&e_2\\
	\hline
	e_1&\alpha e_1& e_2\\
	\hline
	e_2&0&0
	\end{array}$$
	Let $\alpha\ne0$ be fixed and $\nabla^+=\nabla(\alpha)$ be the affine connection determined by this product. An easy calculation shows that a linear basis of the space of complete infinitesimal affine transformations of $\mathsf{Aff}(\mathbb{R})$ relative to $\nabla^+$ is 
	\begin{align*}
		C_1=\frac{1}{\alpha}x\frac{\partial}{\partial x}, && C_2=\frac{1}{\alpha}x^\alpha\frac{\partial}{\partial y}, && C_3=y\frac{\partial}{\partial y},
		& &\mbox{and}&& C_4=\frac{\partial}{\partial y}.
	\end{align*}
	It can be verified that the product determined by $\nabla^+$ on these vector fields is given by
	\begin{align*}
		\begin{array}{c|c|c|c|c}
			&C_1&C_2&C_3&C_4\\
			\hline
			C_1&C_1&C_2&0&0\\
			\hline
			C_2&0&0&C_2&0\\
			\hline
			C_3&0&0&C_3&0\\
			\hline
			C_4&0&0&C_4&0
		\end{array}
	\end{align*}
	By Lemma \ref{L:associativity}, the vector space of infinitesimal affine transformations relative to $\nabla^+$ is an associative 6-dimensional algebra under the product defined by $\nabla^+$. Therefore we get that the vector space with linear basis $C_1,C_2,C_3,C_4$ and the product  on the table above is an associative 4 dimensional subalgebra.  
	This 4-dimensional algebra is the associative envelope of $(\mathfrak{g},\cdot).$ \end{example}

In the following example we use Theorem \ref{T:envelopeofaflataffinemanifold} to naturally construct a linear Lie group whose Lie algebra is the Lie algebra $Env(G,\nabla^+)_-$, with $(G,\nabla^+)$ a real finite dimensional flat affine Lie group. 


\begin{example} One can easily verify that 
\[ G'=\left\{\left(\begin{matrix}1+\beta_1&\beta_2&0\\0&1+\beta_3&0\\\beta_1/\alpha&\beta_2/\alpha+\beta_4&1\end{matrix}\right)\bigg|\ \beta_1\ne-1\ \text{and}\ \beta_3\ne-1\  \right\}\]  
is an enveloping Lie group of the flat affine Lie group $(G,\nabla^+)$ of the previous example.
\end{example}

We finish the section with a more general example.

\begin{example}  Let  $G=GL_n(\mathbb{R})$ endowed with the flat affine bi-invariant connection $D$ determined by composition of linear endomorphisms. Given the local coordinates $\left[x_{ij}\right]$ with $i,j=1,\dots,n$, it is easy to check that  linear bases of left and right invariant vector fields are given by
\[ E_{rs}^+=\sum_{i=1}^n x_{ir}\frac{\partial}{\partial x_{is}}\quad \text{and}\quad E_{rs}^-=\sum_{i=1}^n x_{si}\frac{\partial}{\partial x_{ri}} ,\]
with $r,s=1,\dots,n$. The group $\mathsf{Aff}(G,D)$ is of dimension $2n^2-1$ (\cite{BM})  and the Lie bracket of its Lie algebra is the  bracket of vector fields on $G$.  Using the product determined by $D$ on  left and right invariant vector fields one gets 
\[ D_{E_{pq}^+}E_{rs}^-= x_{sp}\frac{\partial}{\partial x_{rq}},\quad\text{for all}\quad p,q,r,s=1,\dots,n. \]
It follows that an enveloping Lie group of $G$ is $GL_{n^2}(\mathbb{R})$.
\end{example}

The following remark describes the algorithm to compute the associative envelope of a finite dimensional left symmetric algebra.

\begin{remark} Given a finite dimensional real or complex left symmetric algebra $A$, do as follows. 

 Find, applying Lie's third theorem, the connected and simply connected Lie group $G(A)$ of Lie algebra $A_-$ of commutators of $A$.  This group is endowed with a flat affine left invariant connection $\nabla^+$ determined by the product on $A$. 
 
Compute the associative subalgebra $\mathcal{B}$ of $\mathfrak{a}(G(A),\nabla^+)$ generated by  the  space $\mathsf{aff}(G(A),\nabla^+)$.

Finally, the associative envelope $Env(A)$ of $A$ is $\mathcal{B}^{op}$.
\end{remark}

To finish the section, let us pose the following interesting problem.\medskip 

\noindent 
\textbf{Open Problem.} To find an algebraic method to determine the associative envelope, in the sense of the Definition \ref{D:associativeenvelope}, of a real or complex finite dimensional left symmetric algebra.
 
\section{Affine Transformation Groups endowed with a flat Affine or Projective bi-invariant structure} \label{S:studyofaffinetransformationgroups}
We start with the following result that completes and generalizes Theorem 13. in \cite{MSG}. 
\begin{theorem}  \label{T:geodesicallycompletecase}  Let $(M,\nabla)$ be  flat affine and connected. If $\nabla$ is geodesically  complete, then $\mathsf{Aff}(M,\nabla)$ admits a flat affine bi-invariant structure.
	
Moreover $\mathsf{Aff}(M,\nabla)$ contains a normal Lie subgroup $T(M,\nabla)$ so that
\[ Id\longrightarrow T(M,\nabla)\longrightarrow \mathsf{Aff}(M,\nabla)\longrightarrow {}_{T(M,\nabla)}\backslash \mathsf{Aff}(M,\nabla)\longrightarrow Id\]
is a split exact sequence of Lie groups endowed with flat affine bi-invariant structures.
\end{theorem}
\begin{proof} Let $\widehat{M}$ be the universal covering of $M$ with covering map $p:\widehat{M}\longrightarrow M$  and $\widehat{\nabla}$ the connection on $\widehat{M}$ pullback  of $\nabla$ by $p$. From Ehresmann's Developing theorem (see \cite{E}, for more details see \cite{Sh}), there exists an affine immersion $D:\widehat{M}\rightarrow\mathbb{R}^n$ and a Lie group homomorphism $A:\mathsf{Aff}(\widehat{M},\widehat{\nabla})\rightarrow \mathsf{Aff}(\mathbb{R}^n,\nabla^0)$ so that the following diagram commutes 
\[ \xymatrix{ \widehat{M}\ar[d]^F \ar[r]^D&\mathbb{R}^n\ar[d]^{A(F)} \\ \widehat{M}\ar[r]^D &\mathbb{R}^n  }     \]
for all $F\in \mathsf{Aff}(\widehat{M},\widehat{\nabla})$. 

Now, if $F\in Ker A$, that is, $A(F)=Id_{\mathbb{R}^n}$, as the diagram is commutative and $\nabla$ is complete, we get that $F=Id_{\widehat{M}}$. Moreover, the completeness of $\nabla$ also implies that $\dim(\mathsf{Aff}(M,\nabla))=\dim(\mathsf{Aff}(\widehat{M},\widehat{\nabla}))=\dim(\mathsf{Aff}(\mathbb{R}^n,\nabla^0))$ (see \cite{JM2}). Consequently $A$ is an isomorphism of Lie groups.

On the other hand, from Theorem 13. in \cite{MSG}, the group $\mathsf{Aff}(\mathbb{R}^n)$ admits a flat affine bi-invariant structure $\nabla'$, inducing the usual connection $\nabla^0$ of $\mathbb{R}^n$ and the connection on $GL_n(\mathbb{R})$, given by composition of linear endomorphisms, so that the canonical sequence
\[  0\rightarrow\mathbb{R}^n\rightarrow \mathsf{Aff}(\mathbb{R}^n)\rightarrow GL_n(\mathbb{R})\rightarrow Id \]
is a split sequence of  Lie groups admitting flat affine bi-invariant structures.

 Hence transporting $\nabla'$ by means of $A$, we get that $\mathsf{Aff}(\widehat{M},\widehat{\nabla})$ admits a flat affine bi-invariant structure. As $\mathsf{aff}(\widehat{M},\widehat{\nabla})\cong \mathsf{aff}(M,\nabla)$, the Lie group $\mathsf{Aff}(M,\nabla)$ admits a flat affine bi-invariant structure $\widetilde{\nabla}$ with the suitable properties.
\end{proof}  

The following result is a direct consequence  of this  theorem.

\begin{corollary} Let $(G,\nabla^+)$ be a flat affine Lie group and let $(\mathfrak{g},\cdot)$ be the left symmetric algebra determined by $\nabla^+$ on $\mathfrak{g}=$Lie$(G)$. If $\ \nabla^+$ is geodesically complete then the  associative envelope of $(\mathfrak{g},\cdot)$ is isomorphic to $\mathsf{aff}(\mathbb{R}^n,\nabla^0)$ with the product given by composition of affine endomorphims. In fact, the vector space $\mathfrak{a}(G,\nabla^+)$  with the product $XY:=\nabla^+_XY$ is an associative algebra isomorphic to  $\mathsf{aff}(\mathbb{R}^n,\nabla^0)$.\\	
Moreover, if $(M,\nabla)$ is a connected complete flat affine manifold, $\mathfrak{a}(M,\nabla)$  with the product $XY:=\nabla_XY$ is an associative algebra isomorphic to  $\mathsf{aff}(\mathbb{R}^n,\nabla^0)$ with the product of composition of affine endomorphisms
\end{corollary}

\begin{remark} The following two connections on $\mathsf{Aff}(M,\nabla)$ agree. The  connection $\widetilde{\nabla}$, as in the proof of Theorem \ref{T:geodesicallycompletecase}, and   the connection $\nabla^1$ determined by the associative product induced from $\nabla$ on infinitesimal affine transformations, i.e.,
	\[  X Y:=\nabla_XY. \]
\end{remark}

If $\mathsf{Aff}_x(M,\nabla)$ denotes the affine transformation of $(M,\nabla)$ fixing $x\in M$, we have

\begin{corollary}
Let $(M,\nabla)$ be a flat affine manifold. If $\dim \mathsf{Aff}_x(M,\nabla)\geq n^2$ for each $x\in M.$ Then $\mathsf{Aff}(M,\nabla)$ is endowed with a flat affine bi-invariant structure as that of   Theorem \ref{T:geodesicallycompletecase}.
\end{corollary}
\begin{proof} The result follows by recalling that the only obstruction to the completeness of $\nabla$ is the existence of a point $x\in M$ so that $\dim \mathsf{Aff}_x(M)<n^2$ (see \cite{BM} and \cite{T}). 
\end{proof} 

%

 The following proposition is a consequence of the results presented in the previous section.

\begin{proposition}
Let $\nabla$ be a flat affine connection on $M.$ If $M$ is a compact manifold then the Lie group $\mathsf{Aff}(M,\nabla)$ is endowed with a flat affine bi-invariant connection $\widetilde{\nabla}$ inherited from $\nabla$. 
\end{proposition}
\begin{proof} 
	Since $\nabla$ is flat affine, it follows from \eqref{Eq:infinitesimalaffinetransformationsonflataffineconnections} that Product \eqref{Eq:productinducedbyaconnection} induces an associative product on $\mathfrak{a}(M,\nabla)$. As $M$ is a compact manifold every vector field on $M$	is complete, in particular every infinitesimal affine transformation of $(M,\nabla)$ is complete. Hence the Lie algebra $\mathfrak{a}_c(M,\nabla)=\mathfrak{a}(M,\nabla)$ of the Lie group $\mathsf{Aff}(M,\nabla)$ has an associative product given by $XY=\nabla_XY$ with $X,Y\in\mathfrak{a}(M,\nabla)$. Therefore,   from Theorem \ref{T:associativeimplybiinvariantstructures}, there exists  a flat affine bi-invariant structure on $\mathsf{Aff}(M,\nabla)$ determined by $\nabla$. 
\end{proof}

\begin{corollary} Let $G$ be a Lie group endowed with a  flat affine left invariant connection  and let $D$ be a discrete cocompact subgroup  of $G$. The group of affine transformations of the flat affine compact manifold $M=G/D$ has a flat affine bi-invariant structure.  
\end{corollary}

The following result follows from a theorem by Yano (see \cite{Y}, see also Corollary 3.9, page 244 in \cite{KN}).

\begin{corollary} The group of isometries $\mathfrak{I}(M,g)$ of a compact flat Riemannian manifold $(M,g)$, admits a flat affine bi-invariant structure as that of Theorem \ref{T:geodesicallycompletecase}.
\end{corollary}

Let us recall that another interesting case where there is a positive answer to Medina's question is the following (see Theorem 18 in \cite{MSG}).

\begin{theorem} Let $G=GL_n(\mathbb{K})$ with $\mathbb{K}=\mathbb{R}$ or the quaternions $\mathbb{H}$ and $\nabla^+$ be the flat affine bi-invariant connection on $G$ determined by composition of linear endomorphisms of $\mathbb{K}^n$. Then the group $\mathsf{Aff}(G,\nabla^+)$ admits a flat projective bi-invariant structure.
\end{theorem}

\section{  Appendix}

\subsection{Miscellaneous Results} In this subsection we state some results, without proof, consequence of the previous work. The reader can easily verify them.


\begin{proposition} \label{P:envelopingcontainthealgebra} Let $(G,\nabla^+)$ be a flat affine Lie group of finite dimension and $\mathfrak{a}(G,\nabla^+)$ (respectively $\mathsf{aff}(G,\nabla^+)$)  be the real vector space of infinitesimal affine transformations of $(G,\nabla^+)$ (respectively  complete). Then we have the following:
	\begin{enumerate}
		\item[(a)] $\mathfrak{a}(G,\nabla^+)$ with the product defined by $\nabla^+$ is an associative algebra  of dimension at most $n^2+n$. 
		\item[(b)] Let $G'$ be an enveloping Lie group of $(G,\nabla^+)$, then $G$ is a Lie subgroup of $G'$ and $\mathsf{Aff}(G,\nabla^+)$ is a Lie subgroup of $G'^{op}$. Moreover $G'$  is equipped with a flat affine bi-invariant structure $\nabla'$. The Lie bracket of the Lie algebra of $G'^{op}$ is given by the commutator of the product determined by $\nabla^+$. 
	\end{enumerate} 
\end{proposition}

The next result is interesting in its own.

\begin{lemma} Consider a connected manifold  $M$ endowed with a flat affine connection $\nabla$. 
	\begin{enumerate} 
		\item[i.] Any set $S$ of infinitesimal affine transformations of $M$ relative to $\nabla$ determines a unique real finite dimensional associative algebra $\mathcal{A}(S)$ with a product defined by $\nabla$.
		\item[ii.] The Lie algebra of commutators $\mathcal{A}(S)_-$ is of dimension less than or equal to $n^2+n$.
	\end{enumerate}
\end{lemma}

\begin{proposition} Let $(G,\nabla^+)$ be a flat affine Lie group. The Lie group $\mathsf{Aff}(G,\nabla^+)$ is locally isomorphic to a Lie subgroup of an enveloping Lie group $Env(G,\nabla^+)$ endowed with a flat affine bi-invariant structure induced by $\nabla^+$. 
\end{proposition}


\begin{proposition} Let $(G,\nabla^+)$ be a flat affine Lie group and $H$ a Lie group endowed with a flat bi-invariant structure. If $\rho:G\longrightarrow Int(H)$ is a representation of Lie groups then the Lie group $H\rtimes_\rho G$ is endowed with a natural flat affine left invariant structure so that the sequence
	\[ \{\epsilon\}\longrightarrow H\longrightarrow H\rtimes_\rho G\longrightarrow G\longrightarrow \{\epsilon\}\]
	is an exact sequence of flat affine Lie groups. 
\end{proposition}

\subsection{Case $\mathsf{Aff}(\mathbb{R})_0$}
\noindent This subsection is devoted to the study of enveloping Lie groups of the affine group $(\mathsf{Aff}(\mathbb{R})_0,\nabla^+)$ where $\nabla^+$ is any of the flat affine left invariant linear connections listed in Section 3 of \cite{MSG}. 



Recall that the underlying Lie algebra $\mathfrak{a}(G,\nabla^+)_-$ of the associative algebra $\mathfrak{a}(G,\nabla^+)$  of infinitesimal affine transformations of a flat affine Lie group $(G,\nabla^+)$ contains the Lie algebra of right invariant vector fields. Hence, in view of Proposition  \ref{P:envelopingcontainthealgebra}, the associative envelope of $(\mathfrak{g},\cdot)$ is isomorphic to the associative envelope of $(\mathfrak{g}^{op},\cdot^{op}),$ where $\cdot^{op}$ is the product opposite to the product on $\mathfrak{g}$.


\begin{example} \label{R:continuation} Consider the product on $\mathsf{aff}(\mathbb{R})=span\{e_1,e_2\}$   defined by Equation \eqref{Eq:productodeterminadopornablacasoR1} and let $\nabla^+$ be the corresponding flat affine left invariant connection on $\mathsf{Aff}(\mathbb{R})$. The associative envelope $Env(\mathsf{aff}(\mathbb{R}),\cdot)$ of the left symmetric algebra $\mathsf{aff}(\mathbb{R})$ under this product is the algebra with linear basis 
	$$e_1^-,e_2^-,C_3,C_4,C_5$$ with produt given by the opposite of the product  in Table \ref{Tab:tabla1}.
	
	The  corresponding connected and simply connected enveloping  Lie group of $(\mathsf{Aff}(\mathbb{R}),\nabla^+)$  is given by  $\mathcal{H}_3\rtimes_{\rho} \mathbb{R}^2 $, where $\mathcal{H}_3$ is the 3-dimensional Heisenberg Lie group and  $\rho$ is given in Equation \ref{Eq:actiononheisenberg}.
\end{example}




Let $\alpha$ be a real parameter and
$\nabla_1(\alpha)$, $\nabla_2(\alpha)$, $\nabla_1$, or $\nabla_2$ be the flat affine left invariant connections  on $\mathsf{Aff}(\mathbb{R})_0$ described in \cite{MSG} (page 198).  In the following proposition, whose proof is left to the reader, $\nabla^+$ denotes anyone  of these connections.


\begin{proposition} Let $G=\mathsf{Aff}(R)_0$ endowed with the connection $\nabla^+$ and $(\mathfrak{g},\cdot)$ the corresponding left symmetric algebra, then we have
	\begin{enumerate}[(a)]
		\item  The algebra $\mathsf{aff}(G,\nabla^+)$ is an associative algebra with the product determined by $\nabla^+$. The corresponding underlying Lie algebra, $\mathsf{aff}(G,\nabla^+)_-$, contains $\mathfrak{g}$ as a Lie subalgebra, therefore $\mathsf{aff}(G,\nabla^+)$ is the associative envelope of  $(\mathfrak{g},\cdot)$.
		\item The associative  product on $\mathsf{aff}(G,\nabla^+)$ given in i.,  naturally determines a flat affine bi-invariant structure $\nabla$ on $\mathsf{Aff}(G,\nabla^+)$.
		\item The connection $\nabla_1$ is bi-invariant and it induces a flat affine bi-invariant structure $\nabla'$ on the opposite Lie group $G^{op}$.
		\item The flat affine Lie group $\mathsf{Aff}(G,\nabla^+)$ contains $G$ as a normal Lie subgroup which is totally geodesic  relative to $\nabla$, the Lie group $\mathsf{Aff}(G,\nabla^+)/G$ is isomorphic to $G^{op}$ and the natural sequence 
		\[ 1\longrightarrow (G,\nabla_1)\longrightarrow (\mathsf{Aff}(G,\nabla^+),\nabla)\longrightarrow (G^{op},\nabla')\longrightarrow 1\]
		is an exact sequence of affine Lie groups where the affine structures are bi-invariant and  $\nabla'$ is determined by the product $\cdot^{op}$ opposite to the product on $\mathfrak{g}$. 	\end{enumerate}
\end{proposition}



\subsection{Enveloping group of $(\mathsf{Aff}(\mathbb{R}),\nabla_3)$ as a group of transformations of the bundle of linear frames $L(\mathsf{Aff}(\mathbb{R}))$ of $\mathsf{Aff}(\mathbb{R})$} 
A linear basis for the Lie algebra of infinitesimal affine transformations of $M=\mathsf{Aff}(\mathbb{R})$ relative to $\nabla_3$ is given in Example \ref{Ex:productoR1}. The associative envelope of the corresponding left symmetric algebra is generated by the first five elements in this basis. The corresponding affine transformations of $L(M)$ determined by these vector fields, in local coordinates $(x,y,X_{11},X_{12},X_{21},X_{22})$, are given by 
\begin{align*}\phi_{1,t}&=(xe^t,ye^t,X_{11}e^t,X_{12}e^t,X_{21}e^t,X_{22}e^t),\\ 
\phi_{2,t}&=(x,y+t,X_{11},X_{12},X_{21},X_{22})\\
\phi_{3,t}&=\left(\sqrt{2t+x^2},y,\frac{xX_{11}}{\sqrt{2t+x^2}},\frac{xX_{12}}{\sqrt{2t+x^2}},X_{21},X_{22}\right)\\
\phi_{4,t}&=\left(\sqrt{2yt+x^2},y,\frac{xX_{11}+tX_{21}}{\sqrt{2yt+x^2}},\frac{xX_{12}+tX_{22}}{\sqrt{2yt+x^2}},X_{21},X_{22}\right)\\
\phi_{5,t}&=\left(\sqrt{(x^2+y^2)e^{2t}-y^2},y,\frac{(xX_{11}+yX_{21})e^{2t}-yX_{21}}{\sqrt{(x^2+y^2)e^{2t}-y^2}},\frac{(xX_{12}+yX_{22})e^{2t}-yX_{22}}{\sqrt{(x^2+y^2)e^{2t}-y^2}},X_{21},X_{22}\right)
\end{align*}

Therefore, the enveloping Lie group of $\mathsf{Aff}(\mathbb{R})$ is the group generated by $\{\phi_{1,t},\phi_{2,t},\phi_{3,t},\phi_{4,t},\phi_{5,t}\}$. These generators are subject to relations, we exhibit some of them
\begin{align*}  [\phi_{1,t},\phi_{2,t}]&=\phi_{2,s(e^t-1)}, &
[\phi_{1,t},\phi_{3,t}]&=\phi_{2,s(e^{2t}-1)},& 
[\phi_{1,t},\phi_{4,t}]&=\phi_{2,s(e^{t}-1)},\\
[\phi_{2,t},\phi_{4,t}]&=\phi_{3,s(e^{t}-1)}, 
&
[\phi_{3,t},\phi_{5,t}]&=\phi_{2,s(e^{2t}-1)}, &
\end{align*}
where $[\  ,\ ]$ denotes the commutator.

\noindent\textbf{Acknowledgment}

We are very grateful to Martin Bordemann for his suggestions and comments which were  helpful to improve this manuscript.


\end{document}